\documentclass[10pt]{amsart}
\usepackage[english]{babel}
\usepackage{amsmath}
\usepackage{amsfonts}
\usepackage{enumerate}
\usepackage{graphicx}
\usepackage{color}
\usepackage[percent]{overpic}
\usepackage[top=1.62in, bottom=1.3in, left=1.75in, right=1.75in]{geometry}

\newtheorem{thm}{Theorem}[section]

\newtheorem{lem}[thm]{Lemma}
\newtheorem{prop}[thm]{Proposition}

\newtheorem*{claim}{Claim}

\theoremstyle{remark}
\newtheorem{defn}[thm]{Definition}
\newtheorem{rem}[thm]{Remark}
\newtheorem{exam}[thm]{Example}

\numberwithin{equation}{section}

\newcommand{\norm}[1]{\lVert #1 \rVert^2}

\newcommand{\spin}{\ifmmode{\rm Spin}\else{${\rm spin}$\ }\fi}
\newcommand{\spinc}{\ifmmode{{\rm Spin}^c}\else{${\rm spin}^c$}\fi}
\newcommand{\spinct}{\mathfrak t}
\newcommand{\spincs}{\mathfrak s}

\newcommand{\Z}{\mathbb{Z}}
\newcommand{\Q}{\mathbb{Q}}

\DeclareMathOperator*{\Char}{Char}
\DeclareMathOperator*{\PD}{PD}
\DeclareMathOperator*{\coker}{coker}

\begin{document}

\title[Sharp 4-manifolds and the Alexander polynomial]{Surgeries, sharp 4-manifolds and the Alexander polynomial}%
\author{Duncan McCoy}%
\date{}%


\begin{abstract}
Work of Ni and Zhang has shown that for the torus knot $T_{r,s}$ with $r>s>1$ every surgery slope $p/q \geq \frac{30}{67}(r^2-1)(s^2-1)$ is a characterizing slope. In this paper, we show that this can be lowered to a bound which is linear in $rs$, namely, $p/q\geq \frac{43}{4}(rs-r-s)$. The main technical ingredient in this improvement is to show that if $Y$ is an $L$-space bounding a sharp 4-manifold which is obtained by $p/q$-surgery on a knot $K$ in $S^3$ and $p/q$ exceeds $4g(K)+4$, then the Alexander polynomial of $K$ is uniquely determined by $Y$ and $p/q$. We also show that if $p/q$-surgery on $K$ bounds a sharp 4-manifold, then $S^3_{p'/q'}(K)$ bounds a sharp 4-manifold for all $p'/q'\geq p/q$.
\end{abstract}
\maketitle

\section{Introduction}
A knot $K\subseteq S^3$ is said to be an $L$-space knot if there is $p/q\in \mathbb{Q}$ such that the 3-manifold $S_{p/q}^3(K)$ obtained by $p/q$-surgery on $K$ is an $L$-space. It is known that the Heegaard Floer homology of an $L$-space obtained in this way is determined by the Alexander polynomial of $K$ and the surgery slope. We show that under certain circumstances, the Alexander polynomial is determined by the surgery coefficient and the resulting manifold.  A sharp 4-manifold is one whose intersection form determines, in a sense made precise later, the Heegaard Floer homology $d$-invariants of its boundary. Throughout this paper, we will always assume that the intersection form of a sharp 4-manifold is negative-definite.
\begin{thm}\label{thm:Alexuniqueness}
Suppose that for some $p/q>0$, there are knots $K,K'\in S^3$ such that $S^3_{p/q}(K)=S^3_{p/q}(K')$ is an $L$-space bounding a sharp 4-manifold. If $p/q\geq 4g(K)+4$, then
\[\Delta_K(t)=\Delta_{K'}(t) \text{ and } g(K)=g(K').\]
\end{thm}
The most obvious limitation of Theorem~\ref{thm:Alexuniqueness} is the required existence of a sharp 4-manifold. It turns out that given one sharp 4-manifold bounding $S^3_{p/q}(K)$, we can construct one bounding $S^3_{p'/q'}(K)$ for any $p'/q'\geq p/q$.
\begin{thm}\label{thm:sharpextension}
If $S^3_{p/q}(K)$ bounds a sharp 4-manifold for some $p/q>0$, then $S^3_{p'/q'}(K)$ bounds a sharp 4-manifold for all $p'/q'\geq p/q$.
\end{thm}
Owens and Strle show that if $S^3_{p/q}(K)$ bounds a negative-definite 4-manifold $X$, then $S^3_{p'/q'}(K)$ bounds a negative-definite 4-manifold for any $p'/q'\geq p/q$ by taking a certain negative-definite cobordism from $S^3_{p/q}(K)$ to $S^3_{p'/q'}(K)$ and gluing it to $X$ \cite{Owens2012negdef}. We prove Theorem~\ref{thm:sharpextension} by showing that if $X$ is sharp, then this construction results in a sharp 4-manifold.

We apply these results to the problem of finding characterizing slopes of torus knots. We say that $p/q$ is a {\em characterizing slope} for $K$ if $S^3_{p/q}(K)=S^3_{p/q}(K')$ implies that $K=K'$. Using a combination of Heegaard Floer homology and geometric techniques, Ni and Zhang were able to prove the following theorem.
\begin{thm}[Ni and Zhang, \cite{Ni2014characterizing}]\label{thm:NiZhang}
For the torus knot $T_{r,s}$ with $r>s>1$ any non-trivial slope $p/q$ satisfying
\[\frac{p}{q} \geq \frac{30(r^2-1)(s^2-1)}{67}\]
is a characterizing slope.
\end{thm}
Their argument requires a bound on the genus of any knot $K$ satisfying $S^3_{p/q}(K)=S^3_{p/q}(T_{r,s})$. Since $S^3_{p/q}(T_{r,s})$ is an $L$-space bounding a sharp 4-manifold for $p/q\geq rs-1$, we can apply Theorem~\ref{thm:Alexuniqueness} to obtain the equality $g(K)=g(T_{r,s})$, whenever $p/q\geq 4g(K)+4$. This allows us to lower their quadratic bound to one which is linear in $rs$.
\begin{thm}\label{thm:charslopes}
For the torus knot $T_{r,s}$ with $r>s>1$ any non-trivial slope $p/q$ satisfying
\[\frac{p}{q} \geq 10.75(2g(T_{r,s})-1)=\frac{43}{4}(rs-r-s).\]
is a characterizing slope.
\end{thm}

\subsection*{Acknowledgements}
The author would like to thank his supervisor, Brendan Owens, for his helpful guidance. He would also like to acknowledge the influential role of ideas from Gibbons' paper \cite{Gibbons2013deficiency} in the proof of Theorem~\ref{thm:sharpextension}.

\section{Sharp 4-manifolds}\label{sec:sharp}
The aim of this section is to prove Theorem~\ref{thm:sharpextension}. Let $Y$ be a rational homology 3-sphere. Its Heegaard Floer homology is an abelian group which splits as a direct sum over its \spinc-structures:
\[\widehat{HF}(Y)\cong \bigoplus_{\spinct \in \spinc(Y)}\widehat{HF}(Y,\spinct).\]
Associated to each summand there is a numerical invariant $d(Y,\spinct)\in \mathbb{Q}$, called the {\em $d$-invariant} \cite{Ozsvath2003Absolutely_graded}. If $Y$ is the boundary of a smooth, negative-definite 4-manifold $X$, then for any $\spincs \in \spinc(X)$ which restricts to $\spinct \in \spinc(Y)$ there is a bound on the  $d$-invariant:
\begin{equation}\label{eq:sharpdef}
c_1(\spincs)^2+b_2(X)\leq 4d(Y,\spinct).
\end{equation}
We say that $X$ is {\em sharp} if for every $\spinct \in \spinc(Y)$ there is some $\spincs \in \spinc(X)$ which restricts to $\spinct$ and attains equality in \eqref{eq:sharpdef}. Throughout this paper, every sharp manifold is assumed to be negative-definite.

\subsection{A manifold bounding $S^3_{p/q}(K)$}
\begin{figure}[h]
  \begin{overpic}[width=200pt]{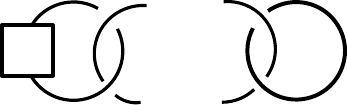}
    \put (15, -4) {$a_1$}
    \put (35, -4) {$a_2$}
    \put (65, -4) {$a_{l-1}$}
    \put (85, -4) {$a_l$}
     \put (5,13) {{\large $K$}}
      \put (50, 13) {{\LARGE $\dots$}}
  \end{overpic}
  \vspace{3mm}
  \caption{A Kirby diagram for $W(K)$ and a surgery diagram for $Y \cong S^3_{p/q}(K)$.}
   \label{fig:kirbydiagram}
\end{figure}
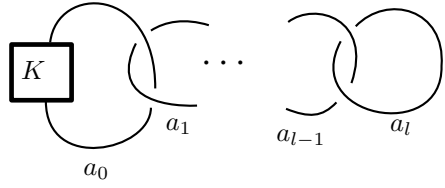

Let $K \subset S^3$ be a knot. For fixed $p/q>0$, with a continued fraction $p/q= [a_0, \dots, a_l]^-$, where
\[
[a_0, \dots, a_l]^-
= a_0 -
    \cfrac{1}{a_1
        - \cfrac{1}{\ddots
            - \cfrac{1}{a_l} } },
\]
and the $a_i$ satisfy
\[
a_i\geq
\begin{cases}
1 & \text{for } i = 0 \text{ or }l\\
2 & \text{for } 0<i <l,
\end{cases}
\]
one can construct a 4-manifold $W$ bounding $Y\cong S^3_{p/q}(K)$ by attaching 2-handles to $D^4$ according to the Kirby diagram given in Figure~\ref{fig:kirbydiagram}. In this paper, all homology and cohomology groups will be taken with integer coefficients. Since $W$ is constructed by attaching 2-handles to a 0-handle, the first homology group $H_1(W)$ is trivial and $H_2(W)$ is a free group with a basis $\{h_0, \dots , h_l\}$ given by the 2-handles. With respect to the basis given by the $h_i$, the intersection form $H_2(W) \times H_2(W) \rightarrow \Z$ is given by the matrix:
\[M =
  \begin{pmatrix}
   a_0  & -1   &        &       \\
   -1   & a_1  & -1     &       \\
        & -1   & \ddots & -1    \\
        &      & -1     & a_l
  \end{pmatrix}.\]
The intersection form of $W$ extends linearly to a $\Q$-valued pairing on  $H_2(W)\otimes \Q$. As $H_1(W)$ is trivial, we may identify $H^2(W)$ with ${\rm Hom}(H_2(W),\Z)$. This allows us to take a basis $\{h_0^*, \dots, h_l^*\}$ for $H^2(W)$, where $h_i^*$ is the function defined by $h^*_i(j_j)=\delta_{ij}$, where $\delta_{ij}$ is the Kronecker delta. Since the intersection pairing on $M$ is non-degenerate, we can identify ${\rm Hom}(H_2(W),\Z)$ with the set
\[
\{\alpha\in H_2(W)\otimes \Q \mid x\cdot \alpha \in \Z \,\text{for all}\, v \in H_2(W)\}\subseteq H_2(W)\otimes \Q,
\]
where the identification is that $\alpha\in H_2(W)\otimes \Q$ corresponds to the function given by $v \mapsto \alpha\cdot v$. Under this identification, the set dual basis element $h_i^*$ is identified with $M^{-1} h_i$. Thus by considering $H^2(W)={\rm Hom}(H_2(W),\Z)$ as a subset of $H_2(W)\otimes \Q$ we obtain a $\Q$-valued pairing on $H^2(W)$ which is expressed by $M^{-1}$ with respect to the basis given by the $h_i^*$.

For an element $\alpha \in H^2(W)$ we use $\norm{\alpha}$ to denote the  of $\alpha$ with respect to this pairing, that is
\[
\norm{\alpha}=\alpha \cdot \alpha= \alpha^T M^{-1} \alpha \in \Q,
\]
where for the last expression we think of $\alpha$ as a row vector written in terms of the basis given by the $h_i^*$.
By considering the long exact sequence of the pair $(W, Y)$, we obtain the short exact sequence:
\[0 \rightarrow H^2(W, Y)\rightarrow H^2(W) \rightarrow H^2(Y) \rightarrow 0.\]
Identifying $H^2(W, Y)$ with $H_2(W)$ via Poincar\'{e} duality gives an injective map $\PD \colon H_2(W)\rightarrow H^2(W)$. When written with respect to the bases for $H_2(W)$ and $H^2(W)$ given by the $h_i$ and the $h_i^*$, respectively, $\PD$ is the linear map corresponding to multiplication by $M$. In particular, the map $\PD$ agrees with the inclusions $H_2(W) \subseteq H^2(W) \subseteq  H_2(W)\otimes \Q$. Consequently, we have isomorphisms:
\begin{equation}\label{eq:H2YcongM}
H^2(Y) \cong \frac{H^2(W)}{\PD(H_2(W))}\cong \coker M.
\end{equation}

Since $H^2(W)$ is torsion-free, the first Chern class defines an injective map
\[c_1:\spinc(W) \rightarrow H^2(W),\]
where the image is the set of characteristic covectors $\Char(W) \subseteq H^2(W)$. A {\em characteristic covector} $\alpha\in H^2(W)$ is one satisfying
\[\alpha \cdot x \equiv x\cdot x \bmod 2, \text{ for all } x\in H_2(W).\]
Since the $h_i$ satisfy $h_i \cdot h_i =a_i$, this allows us to identify $\spinc(W)$ with the set
\[\Char (W) = \{ (c_0, \dots , c_l)\in \Z^{l+1} \,|\, c_i \equiv a_i \bmod 2 \text{ for all } 0\leq i \leq l\}.\]
We will use this identification throughout this section.

Using the map in \eqref{eq:H2YcongM}, which arises from restriction, this allows us to identify the set $\spinc(Y)$ with elements of the quotient
\[\frac{\Char(W)}{2\PD(H_2(W))}.\]
Given $\spincs \in \Char(W)$ we will use $[\spincs]$ to denote its equivalence class modulo $2\PD(H_2(W))$ and the corresponding \spinc-structure on $Y$.

\begin{defn}We say that $\spincs \in \Char(W)$ is {\em short} if it satisfies $\norm{\spincs}\leq \norm{\spincs'}$ for all $\spincs' \in \Char(W)$ with $[\spincs']=[\spincs]$.
\end{defn}
If we let $\spincs=(c_0, \dots, c_l)\in \Char(W)$, then the following calculation will be useful in finding short elements of $\Char(W)$.
\begin{align}\begin{split}\label{eq:normcalc}
\norm{\spincs \pm 2\PD(h_i)} &=(\spincs \pm 2\PD(h_i))M^{-1}(\spincs \pm 2\PD(h_i) )^T  \\
                            &=\norm{\spincs} \pm 4\PD(h_i)M^{-1} \spincs^T + 4\norm{\PD(h_i)}\\
                           &= \norm{\spincs} \pm 4c_i + 4a_i.
\end{split}\end{align}
Note that \eqref{eq:normcalc} immediately implies that $\spincs$ cannot be short unless $|c_i|\leq a_i$ for all $i$.

\subsection{Short representatives in $\Char(W)$}\label{sec:representatives}
Now we identify a set of representatives for $\spinc(S^3_{p/q}(K))$ in $\spinc(W)=\Char(W)$. Following Gibbons, we make the following definitions \cite{Gibbons2013deficiency}.
\begin{defn}\label{def:fulltank}
Given $\spincs =(c_0, \dots, c_l)\in \Char(W)$, we say that it contains a {\em full tank} if there are $0\leq i <j\leq l$, such that $c_i=a_i$, $c_j=a_j$ and $c_k=a_k-2$ for all $i<k<j$.
We say that $\spincs$ is {\em left-full}, if there is $k>0$, such that $c_k=a_k$ and $c_j=a_j-2$ for all $0< j<k$.
\end{defn}
Observe that our definition of left-full does not impose any conditions on $c_0$, and that if $l=0$, then $\Char(W)$ contains no left-full elements.
Let $\mathcal{M}$ denote the set of elements $\spincs =(c_0, \dots, c_l)\in \Char(W)$ satisfying
\[|c_i| \leq a_i, \text{ for all } 0\leq i \leq l\]
and such that neither $\spincs$ nor $-\spincs$ contain any full tanks. Let $\mathcal{C}\subseteq \mathcal{M}$ denote the set of elements $\spincs =(c_0, \dots, c_l)\in \mathcal{M}$ satisfying
\[2-a_i\leq c_i \leq a_i, \text{ for all } 0\leq i \leq l.\]
The set $\mathcal{C}$ will turn out to form a complete set of representatives for $\spinc(Y)$. First we make a careful count of the elements in $\mathcal{C}$.
\begin{lem}\label{lem:spinccount}
Write $p/q$ in the form $p/q=a_0-r/q$, where $q/r=[a_1,\dots, a_l]^-$. We have $|\mathcal{C}|=p$, and for each $c\equiv a_0 \bmod 2$, we have
\[|\{(c_0,\dots, c_l)\in \mathcal{C}\,|\, c_0=c\}|=
\begin{cases}
q   &\text{for } -a_0<c<a_0\\
q-r & c=a_0
\end{cases}
\]
and
\[|\{\spincs=(c_0,\dots, c_l)\in \mathcal{C}\,|\, c_0=c \text{ and $\spincs$ is left full}\}|=
\begin{cases}
r   &\text{for } -a_0<c<a_0\\
0   &\text{for } c=a_0.
\end{cases}
\]
\end{lem}
\begin{proof}
We prove this by induction on the length of the continued fraction $[a_0, \dots , a_l]^-$. When $l=0$, we have $p=a_0$, $q=1$ and $r=0$. In this case,
\[\mathcal{C}=\{-a_0<c\leq a_0 \,|\, c \equiv a_0 \bmod 2\}\]
which clearly has the required properties.

Now suppose that $l>0$. By attaching 2-handles to $B^4$ as in Figure~\ref{fig:kirbydiagram} according to the continued fraction $q/r=[a_1, \dots, a_l]^-$, we obtain a 4-manifold $W'$ whose intersection form is given by
\[M' =
  \begin{pmatrix}
   a_1  & -1   &        &       \\
   -1   & a_2  & -1     &       \\
        & -1   & \ddots & -1    \\
        &      & -1     & a_l
  \end{pmatrix}\]
with respect to the basis given by the 2-handles. The characteristic covectors of $W'$ can be identified with
\[\Char(W')=\{(c_1, \dots , c_l)\,|\, c_i \equiv a_i \bmod 2\}.\]
According to Definition~\ref{def:fulltank}, $(c_1,\dots,c_l) \in \Char(W')$ is left-full if there is $1<j\leq l$ with $c_j=a_j$ and $c_k=a_k-2$ for all $1<k\leq j-1$, and contains a full tank if there are $1\leq i <j\leq l$, such that $c_i=a_i$, $c_j=a_j$ and $c_k=a_k-2$ for all $i<k<j$.

Analogous to $\mathcal{C}\subseteq \Char(W)$, we take $\mathcal{C'}\subseteq \Char(W')$ to be those $(c_1, \dots, c_l)\in \Char(W')$ satisfying $-a_i<c_i\leq a_i$ for all $i$ and not containing any full tanks.
Inductively, we may assume that $|\mathcal{C'}|=q$, and for each $c\equiv a_1 \bmod 2$, we have
\[|\{(c_1,\dots, c_l)\in \mathcal{C'}\,|\, c_1=c\}|=
\begin{cases}
r   &\text{for } -a_1<c<a_1\\
r-r' & c=a_1
\end{cases}
\]
and
\[|\{\spincs'=(c_1,\dots, c_l)\in \mathcal{C'} \,|\, c_1=c \text{ and $\spincs'$ is left full}\}|=
\begin{cases}
r'   &\text{for } -a_1<c<a_1\\
0   &\text{for } c=a_1,
\end{cases}
\]
where $r/r'=[a_2,\dots, a_l]^-$.
For $c\equiv a_0 \bmod 2$ in the range $-a_0< c\leq a_0$, take $\spincs=(c, c_1,\dots, c_l)$. If $c<a_0$, then $\spincs \in \mathcal{C}$ if and only if $(c_1,\dots, c_l)\in \mathcal{C'}$, and $\spincs$ is left-full if and only if $c_1=a_1$ or $c_1=a_1-2$ and $(c_1,\dots, c_l)\in \mathcal{C'}$ is left full. Therefore,
\[|\{(c_0,\dots, c_l)\in \mathcal{C} \,|\, c_0=c\}|=|\mathcal{C'}|=q\]
and
\[|\{\spincs=(c_0,\dots, c_l)\in \mathcal{C} \,|\, c_0=c \text{ and $\spincs$ is left full}\}|= (r-r')+r'=r,\]
for $c<a_0$. If $c=a_0$, then $\spincs \in \mathcal{C}$ if and only if $(c_1,\dots, c_l)\in \mathcal{C'}$ and $\spincs$ contains no full tanks. Equivalently, $\spincs \in \mathcal{C}$ if and only if $(c_1,\dots, c_l)\in \mathcal{C'}$ and $\spincs$ is not left-full. As above, we see that there are $r-r'$ choices of $\spincs'=(c_1,\dots, c_l)\in \mathcal{C'}$ with $c_1=a_1$ and $r'$ choices with $c_1=a_1-2$ and $\spincs'$ is left-full. This shows that
\[|\{(c_0,\dots, c_l)\in \mathcal{C} \,|\, c_0=a_0\}|=q-r\]
and
\[|\{\spincs=(c_0,\dots, c_l)\in \mathcal{C} \,|\, c_0=a_0 \text{ and $\spincs$ is left full}\}|=0,\]
as required. It is then easy to see that $|\mathcal{C}|=(a_0-1)q+q-r=p$.
\end{proof}
It will be useful to consider the following operations on $\Char(W)$
\begin{defn}
Given $\spincs=(c_0, \dots, c_l) \in \Char(W)$ with $c_i=a_i$, we say that $\spincs-2\PD(h_i)$ is the {\em push-down of $\spincs$ at $i$}. Similarly, if $c_i=-a_i$ we say that $\spincs+2\PD(h_i)$ is the {\em push-up of $\spincs$ at $i$}.
\end{defn}
Note that if $\spincs'$ is obtained from $\spincs$ by a push-up or a push-down then $[\spincs]=[\spincs']$ and \eqref{eq:normcalc} shows that $\norm{\spincs}=\norm{\spincs'}$.
Recall that with respect to the bases in use the map $\PD$ is given by multiplication by $M$. Thus we see that if $\spincs'=\spinc\pm 2\PD(h_i)$ is obtained from $\spincs$ by a push-up or push down at $i$ then $\spincs'=(c_0', \dots, c_l')$,
where
\begin{equation}\label{eq:coordchange}
c_k'=
\begin{cases}
\pm a_i &\text{if $k=i$}\\
c_k\mp 2 &\text{if $|k-i|=1$}\\
c_k &\text{if $|k-i|>1$}
\end{cases}
\end{equation}
It follows from \eqref{eq:coordchange} that if $\spincs \in \mathcal{M}$, then any $\spincs'$ obtained from $\spincs$ by a push-up or a push-down is also in $\mathcal{M}$.
\begin{prop}\label{prop:pushuppreservesM}
Suppose that $\spincs'$ is obtained from $\spincs'$ by a push-up or a push-down. Then $\spincs\in \mathcal{M}$ if and only if $\spincs'\in \mathcal{M}$.
\end{prop}

\begin{proof}
Since the effect of a push-up can be undone by a push-down, it suffices to consider the case that $\spincs'=\spincs - 2\PD(h_i)=(c_0',\dots, c_l')$ is obtained by a push-down on $\spincs=(c_0,\dots, c_l)$. By \eqref{eq:coordchange}, this means $c_i=a_i=-c_i'$, $c_k'=c_k$ for $|k-i|>1$ and $c_k'=c_k +2$ for $|k-i|=1$. By carefully comparing the coordinates of $\spincs$ and $\spincs$ with the definition of $\mathcal{M}$, we can establish the following claim.

 \begin{claim}
$\spincs' \not\in \mathcal{M}$ only if $\spincs \not\in \mathcal{M}$.
 \end{claim}
 \begin{proof}[Proof of Claim.]
Suppose that $\spincs' \not\in \mathcal{M}$. This means that either $(a)$ $|c_k'|>a_k$ for some $k$, $(b)$ $\spincs'$ contains a full-tank, or $(c)$ $-\spincs'$ contains a full-tank. We consider each of these in turn.

First, suppose there is $k$ such that $|c_k'|>a_k$. Since $c_i'=-a_i$, we have necessarily have $k\ne i$. If $|k-i|>1$, then $|c_k'|=|c_k|>a_k$, which shows that $\spincs \not\in \mathcal{M}$. If $|k-i|=1$, then $c_k'-2=c_k$, so we either have $c_k=a_k$, giving a full-tank in $\spincs$, or $|c_k|>a_k$. In either case, this means $\spincs\not\in \mathcal{M}$.

Now suppose that $\spincs'$ contains a full-tank, say $c_k'=a_k$, $c_{k'}'=a_{k'}$ and $c_j'=a_j-2$ for $k<j<k'$. Since $c_i'=-a_i$, we must have $k'<i$ or $k>i$. Without loss of generality, assume $k>i$. If $k>i+1$ then $c_k=a_k$, $c_{k'}=a_{k'}$ and $c_j=a_j-2$ for $k<j<k'$, so there is a full-tank in $\spincs\not\in \mathcal{M}$. If $k=i+1$, then $c_i=a_i$, $c_{k'}=a_{k'}$ and $c_j=a_j-2$ for $i<j<k'$, so there is a full-tank in $\spincs$.

Finally, suppose that $-\spincs'$ contains a full-tank, say $c_k'=-a_k$, $c_{k'}'=-a_{k'}$ and $c_j'=2-a_j$ for $k<j<k'$. Since $c_i'=-a_i$ we must have $k'\leq i$ or $k \geq i$. Without loss of generality, assume $k\geq i$. If $k>i+1$ then, as before, then there is clearly a full-tank in $\spincs\not\in \mathcal{M}$. If $k=i+1$, then $c_{i+1}=-a_{i+1}-2$, which shows $\spincs\not\in \mathcal{M}$. If $k=i$, then $c_{i+1}=a_{i+1}$, $c_{k'}=a_{k'}$ and $c_j=a_j-2$ for $i+1<j<k'$, so there is a full-tank in $\spincs$.
\end{proof}
With this claim we can easily finish the proof of the proposition. If $\spincs'$ is obtained from $\spincs$ by a push-down, then $-\spincs$ can be obtained by from $-\spincs'$ by a push-down. Since $\spincs \in \mathcal{M}$ if and only if $-\spincs \in \mathcal{M}$, this shows that $\spincs' \in \mathcal{M}$ if and only if $\spincs \in \mathcal{M}$.
\end{proof}
Now we exhibit the short elements of $\Char(W)$ and show that $\mathcal{C}$ is a set of short representatives for $\spinc(S^3_{p/q}(K))$ (cf. \cite[Lemma~3.3]{Gibbons2013deficiency}.)
\begin{lem}\label{lem:minimisers}
The characteristic vector $\spincs = (c_0, \dots, c_l)\in \Char(W)$ is short if and only if $\spincs \in \mathcal{M}$. For every $\spinct \in \spinc(S^3_{p/q}(K))$, there is a unique $\spincs\in \mathcal{C}$ such that $[\spincs]=\spinct$.
\end{lem}
\begin{proof}
Take $\spincs=(c_0, \dots, c_l)\in \Char(W)$. If $c_i>a_i$ for some $i$, then \eqref{eq:normcalc} shows that $\norm{\spincs - 2 \PD(h_i)}<\norm{\spincs}$. As  $[\spincs- 2 \PD(h_i)]=[\spincs]$, this shows that $\spincs$ is not short unless $c_i\leq a_i$ for all $i$.

Now suppose that $\spincs$ contains a full tank, say $c_j=a_j$, $c_i=a_i$ and $c_k=a_k-2$ for all $i<k<j$. Consider the sequence of elements defined by
\[\spincs_{i-1}=\spincs
\quad \text{and} \quad \spincs_n=\spincs - 2\sum_{k=i}^{n}\PD(h_k) \quad\text{ for $i\leq n <j$.}\]
 We have $[\spincs]=[\spincs_n]$ for each $n$. Since $\spincs_{n}=\spincs_{n-1} - 2\PD(h_{n})$ for $i\leq n< j$, we can use \eqref{eq:coordchange} to see that $\spincs_{n}$ is obtained from $\spincs_{n-1}$ by a push-down and hence for $n<j$, \eqref{eq:normcalc} shows that
 \[\norm{\spincs}=\norm{\spincs_{i-1}}=\dots =\norm{\spincs_{j-1}}.\]
 However in terms of coordinates, we have $\spincs_{j-1}=(c_0', \dots, c_l')$, where $c_j'=a_j+2$. As we have already seen this means that $\spincs_{j-1}$ and, hence $\spincs$, cannot be short.

Since $\spincs$ is short if and only if $-\spincs$ is short, this shows that if $\spincs$ is short, then $|c_i|\leq a_i$ for all $i$ and neither $\spincs$ nor $-\spincs$ contain a full tank. That is, $\spincs$ is short only if $\spincs\in \mathcal{M}$. In order to prove the converse, we need the following claim.

\begin{claim} For every $\spincs \in \mathcal{M}$, there is $\spincs' \in \mathcal{C}$, with $\norm{\spincs}=\norm{\spincs'}$ and $[\spincs]=[\spincs']$.
\end{claim}
\begin{proof}[Proof of Claim]
Given $\spincs\in \mathcal{M}$, we call any $k$ such that $c_k=-a_k$ a {\em trough} for $\spincs$.
Observe that $\spincs\in \mathcal{C}$ if and only if $\spincs$ has no troughs. Given $\spincs_0\in\mathcal{M}$, we may construct a sequence of elements, where $\spincs_{n+1}$ is obtained from $\spincs_n$ by pushing up at a trough if $\spincs_n\notin \mathcal{C}$ and terminating if $\spincs_n\in\mathcal{C}$. Since $\mathcal{M}$ is finite, we can show that this sequence eventually terminates by checking that each of the $\spincs_n$ are distinct. Observe that for any $n$, there are non-negative integers $b_i$ such that $\spincs_n - \spincs_0=\sum_{i=0}^l 2b_i \PD(h_i)$. Since the $\PD(h_i)$ are linearly independent, and the $b_i$ satisfy $\sum_{i=0}^l b_i=n$. This implies that $\spincs_n=\spincs_{n'}$ only if $n=n'$, as required. Thus the sequence terminates with some $\spincs_N\in \mathcal{C}$. Such an element satisfies $[\spincs_N]=[\spincs_0]$ and $\norm{\spincs_0}=\norm{\spincs_N}$, as required.
\end{proof}

Since every $\spinct\in \spinc(S^3_{p/q}(K))$ has a short representative $\spincs'$, which is necessarily in $\mathcal{M}$, the above claim shows that it has a short representative $\spincs \in \mathcal{C}$. However Lemma~\ref{lem:spinccount} shows $|\spinc(S^3_{p/q}(K))|=|\mathcal{C}|=p$, so every element of $\mathcal{C}$ occurs as a short representative for precisely one element of $\spinc(S^3_{p/q}(K))$. It then follows from the above claim that every element of $\mathcal{M}$ must be short.
\end{proof}

We will define one more short set of representatives for $\spinc(S^3_{p/q}(K))$, which we call $\mathcal{F}$. Although the definition of $\mathcal{F}$ may appear unmotivated, one of Gibbons' key ideas in \cite{Gibbons2013deficiency} is that when it comes to working with $d$-invariants, $\mathcal{F}$ is a nicer set of representatives than $\mathcal{C}$ (see Lemma~\ref{lem:evaldi}). We obtain $\mathcal{F}$ from $\mathcal{C}$ as follows.
Take $\spincs =(c_0, \dots, c_l)\in \mathcal{C}$. If $\spincs$ is left-full and $c_0\geq 0$,  then we include $\spincs'=\spincs - 2\sum_{i=1}^k \PD(h_i)$ in $\mathcal{F}$, where $c_k=a_k$ and $c_i=a_i-2$ for $1\leq i <k$. So $\spincs'$ takes the form,
\[\spincs'=
\begin{cases}
(c_0+2, -a_1,2-a_2,\dots, 2-a_{k-1} ,2-a_k, c_{k+1}+2, c_{k+2}, \dots, c_{l}),
& \text{if $k>1$}\\
(c_0+2, -a_1,c_2+2, c_3,\dots, c_{l}), &\text{if $k=1$.}
\end{cases}
\]

Otherwise we include $\spincs$ in $\mathcal{F}$. The following lemma contains the properties of $\mathcal{F}$ that we will require.
\begin{lem}\label{lem:Fproperties}
Every element of $\mathcal{F}$ is short and for each $\spinct \in \spinc(S^3_{p/q}(K))$, there is a unique $\spincs \in \mathcal{F}$ with $[\spincs]=\spinct$. For each $c\equiv a_0 \bmod 2$ and $-a_0<c \leq a_0$, we have
\[|\{(c_0,\dots, c_l)\in \mathcal{F}\,|\, c_0=c\}|=
\begin{cases}
q   &\text{for } c\notin \{0,1\}\\
q-r &\text{for } c\in \{0,1\}
\end{cases}
\]
\end{lem}
\begin{proof} Since $\mathcal{F} \subseteq \mathcal{M}$, every element of $\mathcal{F}$ is short. By construction, for every $\spincs' \in \mathcal{F}$, we either have $\spincs'\in \mathcal{C}$  or there $\spincs'$ is obtained from an element of $\spincs \in \mathcal{C}$ with $[\spincs']=[\spincs]$ by a sequence of push-downs. This shows that $\mathcal{F}$ is a complete set of representatives for $\spinc(S^3_{p/q}(K))$. If $\spincs'=(c_0', \dots, c_l')\ne \spincs=(c_0,\dots, c_l)$, then $c_0'=c_0 +2\geq 2$, and $(c_0, \dots, c_l)$ is left full. Since Lemma~\ref{lem:spinccount} shows that there are $r$ such left-full tuples for each $c_0<a_0$, when we construct $\mathcal{F}$ we increase the first coordinate of $r$ tuples in $\mathcal{C}$ for each $a_0>c_0\geq 0$. This shows that we have the required number for each choice of first coordinate $-a_0<c_0\leq a_0$.
\end{proof}

\subsection{Calculating $d$-invariants}\label{sec:calcdinvariants}
In this section, we set about calculating the $d$-invariants for \spinc-structures on $S^3_{p/q}(K)$ using the sets of representatives given in the previous section.
Since the intersection form on $H^2(W)$ is independent of the choice of the knot $K$, it gives natural choices of correspondences,
\[\spinc(W(K)) \leftrightarrow \Char(W(K)) \leftrightarrow \Char(W(U)) \leftrightarrow \spinc(W(U)),\]
and hence also a choice of correspondence
\begin{equation}\label{eq:homecorrespondence}
\spinc(S_{p/q}^3(K))\leftrightarrow \spinc(S_{p/q}^3(U)).
\end{equation}
Using this we can define $D^{p/q}_K:\spinc(S_{p/q}^3(K)) \rightarrow \mathbb{Q}$ by
\begin{equation*}\label{eq:Ddefinition}
D^{p/q}_K(\spinct)=d(S_{p/q}^3(K),\spinct)-d(S_{p/q}^3(U),\spinct).
\end{equation*}

One can also establish identifications \cite{ozsvath2011rationalsurgery},
\begin{equation}\label{eq:spinccorrespondence}
\spinc(S_{p/q}^3(K)) \leftrightarrow \mathbb{Z}/p\mathbb{Z} \leftrightarrow \spinc(S_{p/q}^3(U)),
\end{equation}
by using relative \spinc-structures on $S^3 \setminus \mathring{\nu} K$. Using this identification, one can similarly define $\widetilde{D}^{p/q}_K: \mathbb{Z}/p\mathbb{Z}\rightarrow \mathbb{Q}$ by the formula
\begin{equation}\label{eq:tildeddef}
\widetilde{D}^{p/q}_K(i):=d(S_{p/q}^3(K),i)-d(S_{p/q}^3(U),i).
\end{equation}
The work of Ni and Wu shows that for $0\leq i \leq p-1$, the values $\widetilde{D}^{p/q}_K(i)$ may be calculated by the formula \cite[Proposition 1.6]{ni2010cosmetic},
\begin{equation}\label{eq:NiWuHV}
\widetilde{D}^{p/q}_K(i)=-2\max\{V_{\lfloor \frac{i}{q} \rfloor},H_{\lfloor \frac{i-p}{q} \rfloor}\},
\end{equation}
where $V_j$ and $H_j$ are sequences of positive integers depending only on $K$, which are non-increasing and non-decreasing respectively. Since we also have that $H_{-i}=V_{i}$ for all $i\geq 0$ [Proof of Theorem~3] \cite{Owens2016Immersed}, this can be rewritten as
\begin{equation}\label{eq:NiWuVonly}
\widetilde{D}^{p/q}_K(i)=-2V_{\min \{\lfloor \frac{i}{q} \rfloor,\lceil \frac{p-i}{q} \rceil\}}.
\end{equation}

When $p/q=n$ is an integer, the correspondence \eqref{eq:spinccorrespondence} can be easily reconciled with the one in \eqref{eq:homecorrespondence}. In this case, $W$ is obtained by attaching a single $n$-framed 2-handle to $D^4$, and the \spinc-structure $c=\Char(W)=\{(c)\,|\, c\equiv n \bmod 2\}$, is labelled by $i \bmod n$, when $n+c \equiv 2i \bmod{2n}$ \cite{Ozsvath2008integersurgery}. It is clear that in this case the correspondences in \eqref{eq:homecorrespondence} and \eqref{eq:spinccorrespondence} are the same. Hence for $c\equiv n \bmod 2$ satisfying $-n<c \leq n$, we have
\begin{equation}\label{eq:integercase}
D^n_K([c])=\widetilde{D}^{n}_K\left(\frac{n+c}{2}\right)=-2V_{\min \{\frac{n+c}{2},\frac{n-c}{2}\}}
=-2V_{\frac{n-|c|}{2}}.
\end{equation}
\begin{rem}
It turns out that the correspondences between $\spinc(S_{p/q}^3(K))$ and $\spinc(S_{p/q}^3(U))$ used in \eqref{eq:homecorrespondence} and \eqref{eq:spinccorrespondence} coincide in general. However, we will not use this fact.
\end{rem}
\begin{lem}\label{lem:spincsum}
If we write $p/q$ in the form $p/q=n-r/q$ with $q>r\geq 0$, then
\[\sum_{\spinct \in \spinc(S_{p/q}^3(K))} D^{p/q}_K(\spinct) = 2rV_{\lfloor\frac{n}{2}\rfloor}+\sum_{\spinct \in \spinc(S_{n}^3(K))}qD^{n}_K(\spinct) \]
\end{lem}
\begin{proof}
Observe that for any $0<\alpha/\beta \in \Q$, the sum 
\[\sum_{\spinct \in \spinc(S_{\alpha/\beta}^3(K))} D^{\alpha/\beta}_K(\spinct)\] is independent of the choices of correspondence between $\spinc(S_{\alpha/\beta}^3(K))$ and $\spinc(S_{\alpha/\beta}^3(U))$, in the sense that we have
\begin{align*}
\sum_{\spinct \in \spinc(S_{\alpha/\beta}^3(K))} D^{\alpha/\beta}_K(\spinct)
&=\sum_{\spinct \in \spinc(S_{\alpha/\beta}^3(K))} d(S_{\alpha/\beta}^3(K),\spinct) - \sum_{\spinct \in \spinc(S_{\alpha/\beta}^3(U))} d(S_{\alpha/\beta}^3(U),\spinct)\\
&=\sum_{i=0}^{p-1} \widetilde{D}^{\alpha/\beta}_K(i).
\end{align*}
This allows us to use \eqref{eq:NiWuVonly} to compute both \[\sum_{\spinct \in \spinc(S_{p/q}^3(K))} D^{p/q}_K(\spinct)\quad\text{and}\quad \sum_{\spinct \in \spinc(S_{n}^3(K))}D^{n}_K(\spinct)\] in terms of the $V_i$. Specifically, if we write $n=2k+\varepsilon$, where $k=\lfloor n/2\rfloor$ and $\varepsilon\in \{\pm 1\}$ we find that
\begin{align*}
\sum_{\spinct \in \spinc(S_{p/q}^3(K))} D^{p/q}_K(\spinct)&=-2\sum_{i=0}^{p-1} V_{\min \{\lfloor \frac{i}{q} \rfloor,\lceil \frac{p-i}{q} \rceil\}}\\
&=-2( qV_0 + \sum_{i=1}^{k-1} 2qV_i + (q+\varepsilon q-r)V_k).
\end{align*}
and
\begin{align*}
\sum_{\spinct \in \spinc(S_{n}^3(K))} D^{n}_K(\spinct)&=-2\sum_{i=0}^{n-1} V_{\min \{i, n-i\}}\\
&= -2(V_0 + 2\sum_{i=1}^{k-1} V_i + (1+\varepsilon) V_k).
\end{align*}
From this, the desired identity follows immediately.
\end{proof}
Since Ozsv{\'a}th and Szab{\'o} have shown that the manifold $-W(U)$ is sharp (see \cite{Ozsvath2003Absolutely_graded}, \cite{Ozsvath2003plumbed} or \cite{Ozsvath2005branched}), for any $\spincs \in \mathcal{M}$ we have
\begin{equation}\label{eq:lensspaced}
d(S_{p/q}^3(U), [\spincs])= \frac{\norm{\spincs}-b_2(W)}{4}=\frac{\norm{\spincs}-l-1}{4}.
\end{equation}
The following lemma allows us to calculate $D^{p/q}_K([\spincs])$ for $\spincs \in \mathcal{C}$.
\begin{lem}[Proof of Lemma~3.10, \cite{Gibbons2013deficiency}]\label{lem:evaldi}
For any $\spincs = (c_0, \dots, c_l)\in \mathcal{F}$, we have
\[D^{p/q}_K([\spincs])=D^{a_0}_K([c_0]) = -2V_{\frac{a_0-|c_0|}{2}}.\]
Consequently, for any $\spincs = (c_0, \dots, c_l)\in \mathcal{C}$, we have
\[D^{p/q}_K([\spincs])=
\begin{cases}
D^{a_0}_K([c_0+2])= -2V_{\frac{a_0-2-c_0}{2}}   &\text{if $0\leq c_0 < a_0$ and $\spincs$ is left full,}\\
D^{a_0}_K([c_0])= -2V_{\frac{a_0-|c_0|}{2}}     &\text{otherwise.}
\end{cases}
\]
\end{lem}
\begin{proof}
Observe that $W$ can be considered as the composition of positive-definite cobordisms,
\[ W_1:\emptyset \rightarrow S_{a_0}^3(K) \text{ and }  W_2:S_{a_0}^3(K) \rightarrow S_{p/q}^3(K),\]
where $b_2(W_1)=1$, $b_2(W_2)=l$ and for any $\spincs \in \spinc(W)$, we have
\[c_1(\spincs)^2= c_1(\spincs|_{W_1})^2+c_1(\spincs|_{W_2})^2.\]
Thus for any $\spincs \in \mathcal{M}$, \eqref{eq:lensspaced} shows that we have
\[
\frac{c_1(\spincs|_{W_2})^2 - l}{4}=
d(S_{p/q}^3(U),[\spincs])-d(S_{a_0}^3(U),[c_0]).
\]
For any $\spincs \in \spinc(W_2)$, which restricts to $\spinct_1$ and $\spinct_2$ on $S_{a_0}^3(K)$ and  $S_{p/q}^3(K)$ respectively, Ozsv{\'a}th and Szab{\'o} show that we get the bound \cite{Ozsvath2003Absolutely_graded}:
\[
\frac{c_1(\spincs)^2 - l}{4}\geq d(S_{p/q}^3(K),\spinct_2)-d(S_{a_0}^3(K),\spinct_1).
\]
Thus, if we take $\spincs|_{W_2}$ for some $\spincs = (c_0, \dots, c_l)\in \mathcal{M}$, then we get
\[
d(S_{p/q}^3(U),[\spincs])-d(S_{a_0}^3(U),[c_0])=
\frac{c_1(\spincs)^2 - l}{4}
\geq d(S_{p/q}^3(K),[\spincs])-d(S_{a_0}^3(K),[c_0]).
\]
Rearranging, this shows that we have
\[D^{p/q}_K([\spincs])\leq D^{a_0}_K([c_0])
\]
Therefore Lemma~\ref{lem:Fproperties} combined with this inequality shows that we have
\begin{align} \begin{split}\label{eq:sumoverF}
\sum_{\spinct \in \spinc(S_{p/q}^3(K))} D^{p/q}_K(\spinct)
&=\sum_{\spincs\in \mathcal{F}}D^{p/q}_K([\spincs])\\
&\leq\sum_{i=1}^{a_0} |\{(c_0,\dots, c_l)\in \mathcal{F}\, |\, c_0=2i-a_0\}| D^{a_0}_K([2i-a_0])  \\
&=2rV_{\lfloor\frac{a_0}{2}\rfloor}+q\sum_{\spinct \in \spinc(S_{a_0}^3(K))}D^{a_0}_{K}(\spinct).
\end{split}\end{align}
However, by Lemma~\ref{lem:spincsum} we must have equality in \eqref{eq:sumoverF}. Since we have termwise inequality in \eqref{eq:sumoverF}, this implies that we must have equality
\[D^{a_0}_K([c_0])= D^{p/q}_K([\spincs]),\]
for each $\spincs=(c_0,\dots, c_l)\in \mathcal{F}$. Using \eqref{eq:integercase} gives the required result in terms of the $V_i$. The formula for $D^{p/q}_K([\spincs])$ when $\spincs \in \mathcal{C}$ follows directly from the construction of $\mathcal{F}$.
\end{proof}

\subsection{Cobordisms}\label{sec:cobords}
Owens and Strle show that if $S^3_{p'/q'}(K)$ bounds a negative-definite manifold, then so does $S^3_{p/q}(K)$ for any $p/q>p'/q'$ by gluing a sequence of negative-definite cobordisms to the original manifold \cite{Owens2012negdef}. We take the same approach to prove Theorem~\ref{thm:sharpextension}.

Let $p/q=[a_0, \dots , a_l]^-$ and $p'/q'=[a_0, \dots, a_l, b_1, \dots, b_k]^-$, where $a_1,b_k\geq 1$ and $a_i,b_j \geq 2$ for all $1\leq i \leq l$ and $1\leq j <k$. Observe that we have $p/q> p'/q'$. Now let $W$ and $W'$ be the 4-manifolds bounding $Y= S_{p/q}^3(K)$ and $Y'=S_{p'/q'}^3(K)$ obtained by attaching 2-handles according to the two given continued fractions as in Figure~\ref{fig:kirbydiagram}. The manifold $W$ is naturally included as a submanifold in $W'$ and $Z=W' \setminus({\rm int}W)$ is a positive-definite cobordism from $S_{p/q}^3(K)$ to $S_{p'/q'}^3(K)$. As in the previous section, we may take a basis for the homology groups $H_2(W)$ and $H_2(W')$ given by the 2-handles and in the same way we may identify $\spinc(W)$ and $\spinc(W')$ with $\Char(H_2(W))$ and $\Char(H_2(W'))$ respectively. We can also define subsets $\mathcal{C}\subseteq \mathcal{M} \subseteq \Char(H_2(W))$ and $\mathcal{C'} \subseteq \mathcal{M'}\subseteq \Char(H_2(W'))$, as in Section~\ref{sec:representatives}.
\begin{lem}\label{lem:spinextension}
Take $\spincs=(c_0,\dots, c_l) \in \mathcal{C}\subseteq \Char(H_2(W))$.
If $(a_0,\dots, a_l,b_1,\dots, b_k)\ne (a_0, 2 , \dots,2, 1)$ or $\spincs\ne (0,\dots,0)$, then there is some short \spinc-structure $\spincs' \in \spinc(W')$, such that $\spincs'|_W=\spincs$ and $D^{p/q}_K(\spincs|_{Y})=D^{p'/q'}_K(\spincs|_{Y'})$.

If $(a_0,\dots, a_l,b_1,\dots, b_k)=(a_0, 2 , \dots,2, 1)$ and $\spincs= (0,\dots,0)$, then there is some short \spinc-structure $\spincs' \in \spinc(W')$, such that $\spincs'|_W=\spincs$ and $D^{p'/q'}_K(\spincs|_{Y'})=-2V_{\frac{a_0-2}{2}}$ and $D^{p/q}_K(\spincs|_{Y})=-2V_{\frac{a_0}{2}}$.
\end{lem}
\begin{proof}
Take $\spincs= (c_0, \dots, c_l) \in \mathcal{C}$. First suppose that one of the following holds:
\begin{enumerate}[(i)]
\item $b_k>1$; or
\item $b_j>2$ for some $1\leq j <k$.
\end{enumerate}
In this case, let $\spincs'$ be the \spinc-structure given by.
\[\spincs'=( c_0, \dots, c_l, 2-b_1, \dots, 2-b_k).\]
It is clear that $\spincs'$ restricts to $\spincs$ on $W$. If $(i)$ holds, then $2-b_k<b_k$ and hence we have $2-b_i<b_i$ for all $1\leq i \leq k$. If $(ii)$ holds, then we have some $j<k$ such that $2-b_j<b_j-2$. In either case, this shows that $\spincs'$ contains no full tanks and that $\spincs'$ is left-full if and only if $(c_1, \dots, c_l)$ is left-full. Therefore, $\spincs' \in \mathcal{C}'$ and by Lemma~\ref{lem:minimisers} and Lemma~\ref{lem:evaldi}, we have that $\spincs'$ is short and that $D^{p/q}_K([\spincs])=D^{p'/q'}_K([\spincs'])$. This proves the lemma if either $(i)$ or $(ii)$ hold.

Therefore we may assume that $b_i=2$ for $1\leq i < k$ and $b_k=1$. We consider the case where there is $0<j\leq l$ such that $c_j>2-a_j$ or $a_j>2$.
In this case, we define
\[\spincs'=( c_0, \dots, c_l, 0, \dots,0, -1).\]
This clearly restricts to $\spincs$ and by Lemma~\ref{lem:minimisers} it is short. It remains to calculate $D^{p'/q'}_K([\spincs'])$. Take $0<t\leq l$ to be maximal such that $a_t>2$ or $c_t>2-a_t$. Since $c_j=0$ and $a_j=2$ for all $l \geq j>t$, we can assume for convenience that $t=l$. For $1\leq i \leq k$, let $h_i'$ denote 2-handle attached with framing $b_i$ in the handle decomposition of $W'$. Consider now the \spinc-structure $\spincs''$ defined by
\begin{align*}
\spincs''&= \spincs'+2\sum_{i=1}^k i\PD(h_i')\\
         &= (c_0,\dots, c_{l-1}, c_l-2, 0 , \dots , 0,1).
\end{align*}
By construction, we have that $[\spincs'']=[\spincs']\in \spinc(S_{p'/q'}^3(K))$, $\spincs''\in \mathcal{C'}$ and $\spincs''$ is left-full if and only if $\spincs$ is left-full. Thus by Lemma~\ref{lem:evaldi}, we have
\[D^{p/q}_K([\spincs])=D^{p'/q'}_K([\spincs''])= D^{p'/q'}_K([\spincs']).\]

Thus it remains only to prove the lemma when $c_j=2-a_j=0$ for all $0<j\leq l$. In this case, we may take
\[
\spincs'=
\begin{cases}
(c_0,0, \dots,0,-1) &\text{if }c_0>0\\
(c_0,0, \dots,0,1) &\text{if }c_0\leq 0.\\
\end{cases}
\]
We have either $\spincs'\in \mathcal{C'}$ or $-\spincs'\in \mathcal{C'}$. In either case, Lemma~\ref{lem:minimisers} and Lemma~\ref{lem:evaldi}, show that it has the required properties. In particular, if $c_0=0$, then $a_0$ is necessarily even and we have \[D^{p'/q'}_K([\spincs'])=-2V_{\frac{a_0-2}{2}}\text{ and } D^{p/q}_K([\spincs])=-2V_{\frac{a_0}{2}},\]
as required.
\end{proof}

\begin{lem}\label{lem:specialcase}
If $S^3_{2n}(K)$ bounds a sharp 4-manifold for $n\geq 1$, then $V_{n}=V_{n-1}=0$.
\end{lem}
\begin{proof}[Proof (sketch).]
Let $\tilde{g}\geq 0$ be minimal such that $V_{\tilde{g}}=0$. Greene shows that if $S_p^3(K)$ is an $L$-space bounding a sharp 4-manifold, then, by Theorem~1.1 of \cite{Greene2010genusbounds},
\[2g(K)-1\leq p-\sqrt{3p+1}.\]
In the proof of this inequality, the $L$-space condition is only required to show that
\[d(S_{p}^3(K),i)-d(S_{p}^3(U),i)\leq 0,\]
for all $i$ with equality if and only if $\min \{p-i,i\}\geq g(K)$. However, since \eqref{eq:NiWuVonly} shows that
\[d(S_{2n}^3(K),i)-d(S_{2n}^3(U),i)\leq 0,\]
for all $i$, with equality if and only if $\min \{2n-i,i\}\geq \tilde{g}$, the same argument shows the bound
\[2\tilde{g}-1\leq 2n-\sqrt{6n+1}.\]
This inequality implies
\[\tilde{g}\leq n+\frac{1}{2}-\frac{1}{2}\sqrt{6n+1}<n-\frac{1}{2},\]
and hence that $V_n=V_{n-1}=0$, as required.
\end{proof}
Recall that $Z=W'\setminus ({\rm int} W)$ is a cobordism from $Y=S_{p/q}^3(K)$ to $Y'=S_{p'/q'}^3(K)$.
\begin{lem}\label{lem:surgcobordsharp}
If $Y'$ is the boundary of a sharp 4-manifold $X'$, then the manifold $X=(-Z)\cup_{Y'} X'$ is a sharp 4-manifold bounding $Y$.
\end{lem}
\begin{proof}
It is clear from the construction that $X$ is negative-definite with $b_2(X)=b_2(Z)+b_2(X')$ and $\partial X =Y$. Together, Lemma~\ref{lem:spinextension} and Lemma~\ref{lem:specialcase} show that for every $\spinct\in \spinc(Y)$, there exists a short $\spincs' \in \spinc(-Z)$ such that $\spincs'|_{Y}=\spinct$ and $D^{p'/q'}_K(\spinct')=D^{p/q}_K(\spinct)$, where $\spincs'|_{-Y'}=\spinct'$. In each case, such a $\spincs'$ can be obtained by restricting the \spinc-structure given in Lemma~\ref{lem:spinextension}. The equality $D^{p'/q'}_K(\spinct')=D^{p/q}_K(\spinct)$ either follows directly from Lemma~\ref{lem:spinextension} or from Lemma~\ref{lem:specialcase} which guarantees that $V_{\frac{a_0}{2}}=V_{\frac{a_0-2}{2}}=0$ when $a_0$ is even. By using \eqref{eq:lensspaced}, we see that such a $\spincs'$ satisfies
\begin{align*}
\frac{c_1(\spincs')^2+b_2(Z)}{4}&=d(S_{p/q}^3(U),\spinct)-d(S_{p'/q'}^3(U),\spinct')\\
    &=(d(Y,\spinct)-D^{p/q}_K(\spinct))-(d(Y',\spinct')-D^{p'/q'}_K(\spinct'))\\
    &=d(Y,\spinct)-d(Y',\spinct').
\end{align*}
Since $X'$ is sharp, there is $\mathfrak{r}\in \spinc(X')$ such that $\mathfrak{r}|_{Y'}=\spinct'$ and
\[
\frac{c_1(\mathfrak{r})^2+b_2(X')}{4}=d(Y',\spinct').
\]
The \spinc-structure $\spincs\in \spinc(X)$ obtained by gluing $\mathfrak{r}$ to $\spincs'$ on $-Z$ satisfies $\spincs|_Y=\spinct$ and
\begin{align*}
\frac{c_1(\spincs)^2+b_2(X)}{4}&=\frac{c_1(\spincs')^2+b_2(Z)+c_1(\mathfrak{r})^2+b_2(X')}{4}\\
    &=(d(Y,\spinct)-d(Y',\spinct'))+d(Y',\spinct')\\
    &=d(Y,\spinct).
\end{align*}
This shows that $X$ is sharp, as required.
\end{proof}
\begin{proof}[Proof of Theorem~\ref{thm:sharpextension}]
Let $a_1, \dots, a_l, b_1, \dots b_k$ be positive integers with $a_1, b_k\geq 1$ and $a_i,b_j \geq 2$ for $i\neq 1$ and $j\neq k$. Then
Lemma~\ref{lem:surgcobordsharp} shows that if $S_{[a_0, \dots, a_l, b_1, \dots, b_k]^-}^3(K)$ bounds a sharp 4-manifold, then so does $S^3_{[a_0, \dots , a_l]^-}(K)$. In particular, if $S_{[a_0, \dots, a_l]^-}^3(K)$ bounds a sharp manifold, then the identities:
\[[a_0, \dots, a_l]^-=[a_0, \dots,a_{l-1}, a_l+1,1]^-=[a_0, \dots,a_{l-1}, a_l+1,2, \dots, 2, 1]^-\]
shows surgeries of the form
\[S_{[a_0, \dots, a_l+1]^-}^3(K) \quad \text{and} \quad S_{[a_0, \dots, a_l+1,2, \dots, 2]^-}^3(K)\]
 also bounds a sharp 4-manifolds. Furthermore, repeated applications of these identities shows that for any sequence of integers $b_1, \dots , b_k$ with $b_1\geq 1$  and $b_i\geq 2$ for $i\geq 2$, we can show that if $S^3_{[a_0, \dots , a_l]^-}(K)$ bounds a sharp manifold, then $S^3_{[a_0, \dots , a_l+b_1, b_2, \dots, b_k]^-}(K)$ bounds a sharp manifold.

  Now if we have rational numbers $p'/q'>p/q$, then for some $m\geq0$, we can write their continued fractions in the forms
\[p/q=[a_1, \dots, a_m, a_{m+1}, \dots, a_{m+k}]^-\]
and
\[p'/q'=[a_1, \dots, a_m, a_{m+1}', \dots, a_{m+k'}']^-,\]
where $a_{m+1}'>a_{m+1}$. 
Now if $S_{p/q}^3(K)$ bounds a sharp 4-manifold, then Lemma~\ref{lem:surgcobordsharp} shows that 
$S_{[a_1, \dots, a_m, a_{m+1}]^-}^3(K)$ bounds a sharp 4-manifold. Since  $a_{m+1}'\geq a_{m+1}+1$, the preceding discussion establishes that  $S_{p'/q'}^3(K)$ bounds a sharp 4-manifold, as required.

\end{proof}

\section{The Alexander polynomial}\label{sec:Alexpolys}
When positive surgery on a knot in $S^3$ bounds a sharp 4-manifold $X$ results of Greene, in the integer and half-integer case \cite{Greene2013Realization, Greene2014_3Braid, Greene2010genusbounds}, and Gibbons, in the general case \cite{Gibbons2013deficiency}, show that the intersection form of $X$ takes the form of a changemaker lattice. In this section, we state the changemaker theorem and derive the properties of changemaker lattices required to prove Theorem~\ref{thm:Alexuniqueness}.

\subsection{Changemaker lattices}
The changemaker condition from which changemaker lattices get their name is the following.

\begin{defn}We say $(\sigma_1, \dots , \sigma_t)$ satisfies the {\em changemaker condition}, if the following conditions hold:
\[0\leq \sigma_1 \leq 1 \text{ and } \sigma_{i-1} \leq \sigma_i \leq \sigma_1 + \dots + \sigma_{i-1} +1,\text{ for } 1<i\leq t.\]
\end{defn}
We give the definition of integer and non-integer changemaker lattices separately, although the two are clearly related.
\begin{defn}[Integer changemaker lattice]\label{def:intCMlattice}
First suppose that $q=1$, so that $p/q>0$ is an integer. Let $f_0, \dots, f_t$ be an orthonormal basis for $\mathbb{Z}^t$. Let $w_0=\sigma_1 f_1 + \dots + \sigma_t f_t$ be a vector such that $\norm{w_0}=p$ and $(\sigma_1, \dots , \sigma_t)$ satisfies the changemaker condition, then
\[L=\langle w_0\rangle^\bot \subseteq \mathbb{Z}^{t+1}\]
is a {\em $p/q$-changemaker lattice}. Let $m$ be minimal such that $\sigma_m>1$. We define the {\em stable coefficients} of $L$ to be the tuple $(\sigma_m, \dots, \sigma_t)$. If no such $m$ exists, then we take the stable coefficients to be the empty tuple.
\end{defn}

\begin{defn}[Non-integer changemaker lattice]\label{def:nonintCMlattice}
Now suppose that $q\geq 2$ so that $p/q>0$ is not an integer. This has continued fraction expansion of the form, $p/q=[a_0,a_1, \dots , a_l]^{-}$, where $a_k\geq 2$ for $1\leq k \leq l$ and $a_0=\lceil \frac{p}{q}\rceil \geq 1$. Now define
\[m_0=0 \text{ and } m_k=\sum_{i=1}^ka_i -k \text{ for } 1\leq k \leq l.\]
Set $s=m_{l}$ and let $f_1, \dots, f_t, e_0, \dots, e_s$ be an orthonormal basis for the lattice $\mathbb{Z}^{t+s+1}$.
Let $w_0=e_0+\sigma_1 f_1 + \dots + \sigma_t f_t,$ be a vector such that $(\sigma_1, \dots, \sigma_t)$ satisfies the changemaker condition and $\norm{w_0}=n$. For $1\leq k \leq l$, define
\[w_k=-e_{m_{k-1}}+e_{m_{k-1}+1}+ \dots + e_{m_{k}}.\]
We say that
\[L=\langle w_0, \dots, w_l\rangle^\bot \subseteq \mathbb{Z}^{t+s+1}\]
is a {\em $p/q$-changemaker lattice}. Let $m$ be minimal such that $\sigma_m>1$. We define the {\em stable coefficients} of $L$ to be the tuple $(\sigma_m, \dots, \sigma_t)$. If no such $m$ exists, then we take the stable coefficients to be the empty tuple.
\end{defn}

\begin{rem}Since $m_k-m_{k-1}=a_k-1$, the vectors $w_0, \dots, w_l$ constructed in Definition~\ref{def:nonintCMlattice} satisfy
\[
w_i\cdot w_j =
  \begin{cases}
   a_j            & \text{if } i=j\\
   -1       & \text{if } |i-j|=1\\
   0        & \text{otherwise.}
  \end{cases}
\]
\end{rem}
Now we are ready to state the changemaker theorem we will use.
\begin{thm}[cf. Theorem~1.2 of \cite{Gibbons2013deficiency}]\label{thm:Gibbons}
Suppose that for $p/q=n-r/q>0$, the manifold $S^3_{p/q}(K)$ bounds a negative-definite, sharp 4-manifold $X$ with intersection form $Q_X$. Then for $N=b_2(X)+l+1$, we have an embedding of $-Q_X$ into $\mathbb{Z}^{N}$ as a $p/q$-changemaker lattice,
\[-Q_X \cong \langle w_0,\dots, w_l \rangle^\bot \subseteq \mathbb{Z}^{N}\]
such that $w_0$ satisfies
\begin{equation}\label{eq:Viformula}
8V_{|i|} = \min_{ \substack{ c\cdot w_0 \equiv 2i-n \bmod 2n \\ c \in \Char(\mathbb{Z}^{N})}} \norm{c} - N,
\end{equation}
for all $|i|\leq n/2$.
\end{thm}
The equation \eqref{eq:Viformula} is not explicitly stated by Gibbons. However, Greene shows that it holds in the case of integer surgeries \cite[Lemma~2.5]{Greene2010genusbounds} and we will deduce it in the general case using the results of Section~\ref{sec:sharp}. We also point out that Theorem~\ref{thm:Gibbons} does not contain the hypotheses on the $d$-invariants of $S^3_{p/q}(K)$ which were present in Gibbons' original statement. These are omitted since it can be shown that they are automatically satisfied (cf. \cite[Section~2]{McCoy2015noninteger}).
\begin{proof}[Proof of \eqref{eq:Viformula}]
Let $W'$ be the positive-definite 4-manifold bounding $S_{p/q}^3(K)$ obtained by attaching 2-handles $h_0,\dots, h_l$ to $S^3$ according to the Kirby diagram in Figure~\ref{fig:kirbydiagram}. This can be decomposed as $W\cup Z$, where $W$ has boundary $S_{p/q}^3(K)$ and is obtained from $D^4$ by attaching a single $n$-framed 2-handle along $K$ in $\partial D^4=S^3$ and $Z$ a cobordism from $S_n^3(K)$ to $S_{p/q}^3(K)$ obtained by 2-handle attachment. The homology group $H_2(W)$ is generated by the class given by gluing the core of the 2-handle to a Seifert surface $\Sigma$.  We will call this generator $[\Sigma]$. Let $X'$ be the closed smooth positive-definite 4-manifold $X'=W'\cup (-X)=W\cup Z\cup (-X)$. This has second Betti number $b_2(X')=b_2(X)+l+1$ and Donaldson's Theorem shows that the intersection form on $H_2(X')$ is diagonalisable, i.e $H_2(X')\cong \mathbb{Z}^{b_2(X')}$ \cite{donaldson1983application}. Let $\sigma \in H_2(W\cup Z\cup (-X))$ be the class given by the inclusion of $[\Sigma]$ into $H_2(X')$. Since Lemma~\ref{lem:surgcobordsharp} shows that $(-Z)\cup X$ is a sharp 4-manifold bounding $S_n^3(K)$, Greene shows that $\sigma$ satisfies \cite[Lemma~2.5]{Greene2010genusbounds}
\[8V_{|i|} = \min_{ \substack{ c\cdot \sigma \equiv 2i-n \bmod 2n \\ c \in \Char(\mathbb{Z}^{b_2(X')})}} \norm{c} - b_2(X'),\]
for all $|i|\leq n/2$.
Since the vector $w_0$ occurring in Theorem~\ref{thm:Gibbons} is precisely the image of $[\Sigma]$ with respect to some choice of orthonormal basis for $H_2(X')$, the above equation gives \eqref{eq:Viformula}, as desired.
\end{proof}
Now we prove that under certain hypotheses the changemaker structure on a lattice is unique.
\begin{rem} Since there are examples of lattices admitting embeddings into $\Z^N$ as changemaker lattices in more than one way, we cannot prove unconditionally that the changemaker structure of a lattice is unique. For example, we have an isomorphism of lattices
\[\langle 4e_0+e_1+e_2+e_3+e_4+e_5\rangle^\bot\cong \langle 2e_0 + 2e_1+ 2e_2+2e_3+2e_4+e_5 \rangle^\bot \subseteq \Z^6.\]
This isomorphism can be seen by observing that both lattices admit a basis for which the bilinear form is given by the matrix
\[
  \begin{pmatrix}
   5   & -1  &      &    &     \\
   -1  & 2   & -1   &    &     \\
       & -1  & 2    & -1 &     \\
       &     & -1   & 2  &-1   \\
       &     &      & -1 & 2   \\
  \end{pmatrix}.
\]
This example is a consequence of the fact that $S^3_{21}(T_{5,4})\cong S^3_{21}(T_{11,2})\cong L(21,4)$.
\end{rem}

\begin{lem}\label{lem:CMuniqueness}
Let $L=\langle w_0, \dots , w_l \rangle^{\bot}\subseteq \mathbb{Z}^{N}$ be a $p/q$-changmaker lattice with stable coefficients $(\rho_1, \dots, \rho_t)$. If $p/q \geq \sum_{i=1}^t \rho_i^2 + 2\rho_t$, then for any embedding $\phi:L\rightarrow \mathbb{Z}^N$ such that
\[\phi(L)=\langle w_0', \dots , w_l' \rangle^{\bot} \subseteq \mathbb{Z}^N\] is a $p/q$-changemaker lattice, there is an automorphism of $\mathbb{Z}^N$ which maps $w_0$ to $w_0'$.
\end{lem}

\begin{proof}
If we write $p/q=n-r/q$, where $0\leq r<q$, then by definition there is a choice of orthonormal basis for $\mathbb{Z}^N$ such that $w_0$ takes the form
\[w_0=
\begin{cases}
\rho_t e_{m+t} + \dots + \rho_1 e_{m+1}+ e_m + \dots + e_1      &\text{if $q=1$ and} \\
\rho_t e_{m+t} + \dots + \rho_1 e_{m+1}+ e_m + \dots + e_1 +e_0 &\text{if } q>1,
\end{cases}
\]
where $m\geq 2\rho_t\geq 4$ and $\norm{w_0}=n$. It follows that $L$ contains vectors $v_2,\dots, v_{m+t}$ defined by
\[ v_k = e_{k-1} - e_k \quad\text{for $2\leq k\leq m$.} 
\]
and
\[v_{m+k}= -e_{m+k}+e_1 + \dots + e_{\rho_k} \quad\text{for $2\leq k\leq t$.}\]
These satisfy $\norm{v_{m+k}}=1+\rho_k$, for $1\leq k \leq t$, and $v_{m+k}\cdot v_{m+l} = \min\{ \rho_l,\rho_k\}=\rho_k$ for $1\leq k<l \leq t$. For $j$ and $k$ satisfying $2\leq k<j \leq m$, we have $\norm{v_k}=\norm{v_j}=2$ and
\[
v_k\cdot v_j=
\begin{cases}
-1 &\text{if } j=k+1\\
0  &\text{otherwise.}
\end{cases}
\]

We will consider the image of these vectors under $\phi$. For $k$ in the range $2\leq k \leq m+t$, let $u_k$ denote the vector $u_k= \phi(v_k)$.

First consider $u_2$. Since $\norm{u_2}=2$, there must be two orthogonal unit vectors, $f_1,f_2$ in $\Z^N$ such that $u_2=-f_2+f_1$. As $\norm{u_3}=2$ and $u_2\cdot u_3=-1$, we can assume, without loss of generality, that there is a unit vector $f_3$ orthogonal to $f_1$ and $f_2$ such that $u_3=-f_3+f_2$. There are two possibilities for $u_4$. We can either have that $(1)$ $u_4=-f_2-f_1$ or that $(2)$ there is a unit vector $f_4$ with $\pm f_4\notin \{f_1,f_2,f_3\}$ such that $u_4=-f_4+f_3$.

If possibility $(1)$ holds, then we have $u_2-u_4=2f_1\in \phi(L)$. Thus $w_i'\cdot (2f_1)=w_i'\cdot f_1=0$ for all $i$, so we have $f_1\in \phi(L)$. However, this is impossible, since there are no unit vectors in $L$ which pair non-trivially with $v_2$. Thus we can conclude that $u_4=-f_4+f_3$.

For the vectors $u_5, \dots, u_m$, which all have norm 2, there is now no choice: there must exist a choice of distinct orthogonal unit vectors $f_1, \dots, f_m$ in $\Z^N$, such that $u_k=-f_k+f_{k-1}$ for each $k$ in the range  $2\leq k \leq m$.

Now we determine the form that $u_{m+j}$ must take for $1\leq j \leq t$. Let $\lambda_j \in \Z$ denote the quantity $\lambda_j=u_{m+j}\cdot f_1$. For $2\leq k \leq m$ we have $v_k\cdot v_{m+j}=0$ for $k\ne  \rho_{j}$ and $v_k\cdot v_{\rho_{j}}=1$. As we have
\[u_k \cdot u_{m+j}=u_{m+j}\cdot f_{k-1} - u_{m+j}\cdot f_{k}
\quad\text{for $2\leq k\leq m$,}\]
this shows that
\[u_{m+j}\cdot f_k =
\begin{cases}
\lambda_j   &\text{for } 1\leq k\leq \rho_{j}\\
\lambda_j-1 &\text{for } \rho_j < k \leq m.
\end{cases}
\]
Computing the norm of $u_{m+j}$ shows that
\begin{equation}\label{eq:vjnorm}
\rho_j+1 =\norm{u_{m+j}}\geq \lambda_j^2 \rho_j + (\lambda_j-1)^2(m-\rho_j).
\end{equation}
Now by assumption we have that $m-\rho_j\geq \rho_j$, showing that at most one of $\lambda_j$ and $\lambda_j+1$ is non-zero. Thus we either have $\lambda_j=1$ or $\lambda_j=1$. Note further that if $\lambda_j=0$, then $m=2\rho_j+1$ or $m=2\rho_j$. The following claim shows that only one of these possibilities for $\lambda_1=0$ can occur.
\begin{claim}
If $\lambda_j=0$, then $m=2\rho_1$.
\end{claim}
\begin{proof}
If $\lambda_j=0$ and $m=2\rho_1+1$, then we obtain equality in \eqref{eq:vjnorm}, showing that $u_j$ takes the form
\[u_j=-(f_{\rho_j+1}+\dots +f_{m}).\]
However, this results in a contradiction when we consider the possibility for the $w_i'$. Since the $w_i'\cdot u_k=0$ for all $i$ and all $k$, we must have $w_i'\cdot f_1=\dots = w_i'\cdot f_m$. However, in the case at hand this also implies $u_j\cdot w_i'=(\rho_j+1)(w_i'\cdot f_1)=0$ for all $i$, i.e $w_i\cdot f_1 =0$. This implies that $f_1\in \phi(L)$, which is a contradiction.
\end{proof}

Thus we see that $u_{m+j}$ may be assumed to be in the form
%
\begin{equation}\label{eq:ujform}
u_{m+j}=
\begin{cases}
-f_{j+m}+ f_1 + \dots + f_{\rho_{j}} &\text{if } \lambda_j = 1\\
f_{j+m} - (f_{\rho_j+1}+\dots + f_m)       &\text{if } \lambda_j =0.
\end{cases}
\end{equation}
for some choices of unit vector $f_{j+m} \notin \{\pm f_1, \dots, \pm f_m\}$.
Observe that if $\lambda_j$ were not equal to $\lambda_1$ for some $j>1$, then we would have
\[|u_{m+1}\cdot u_{m+j}|=f_{j+m}\cdot f_{1+m}\leq 1.\]
This would contradict the fact that
\[u_{m+1}\cdot u_{m+j}=v_{m+1}\cdot v_{m+j} = \rho_1 \quad\text{for $j>1$.}\]
Thus we can conclude that $\lambda_j=\lambda_1$ for all $j$. Thus if we compute $u_{m+k}\cdot u_{m+j}$ for $1\leq j<k\leq t$, we obtain
\[u_{m+k}\cdot u_{m+j}=\rho_j+f_{j+m}\cdot f_{k+m}.\]
As $u_{m+k}\cdot u_{m+j}=v_{m+k}\cdot v_{m+j} = \rho_j$ for all $1\leq j<k\leq t$, this shows that the unit vectors $f_{m+1}, \dots, f_{m+t}$ must all be distinct.

As we are assuming that $\phi$ is an embedding of $L$ into $\Z^N$ as $p/q$-changemaker lattice
\[\phi(L)=\langle w'_0, \dots , w'_l \rangle^{\bot},\]
we have $|w_i'\cdot f|\leq 1$ for any $i\geq 1$ and any unit vector $f\in \Z^{N}$. We also have $\norm{w_0'}=n$.

Let $x$ be a vector in the orthogonal complement of $\phi(L)$. Since $x$ must satisfy $u_k\cdot x=0$ for $2\leq k\leq m+t$, and these $u_k$ take the form given in \eqref{eq:ujform} we have
\[
x\cdot f_k =
\begin{cases}
x\cdot f_1               &\text{for } 1\leq k\leq m \text{ and} \\
\rho_{k-m}(x\cdot f_1)   &\text{for } m+1 \leq k \leq m+t.
\end{cases}
\]
In particular, if $x\cdot f_1\ne 0$, then $|x\cdot f_{m+t}|>1$. Thus we must have $w_i'\cdot f_1=0$, for all $i\geq 1$. However as we have used previously, $f_1\notin \phi(L)$, so we must have $w_0'\cdot f_1 \ne 0$. Thus if we compute the norm of $w_0'$, we arrive at the inequality
\[\norm{w_0'}=n\geq (w_0'\cdot f_1)^2 \left(\sum_{i=1}^t\rho_i^2 + m\right)\geq (w_0'\cdot f_1)^2(n-1).\]
This shows that $w_0'\cdot f_1=\pm 1$. Without loss of generality, we may assume thats that $w_0'\cdot f_1=1$. It follows $w_0'$ must take the form,
\[w'_0=
\begin{cases}
\rho_t f_{m+t} + \dots + \rho_1 f_{m+1}+ f_m + \dots + f_1      &\text{if } q=1, \\
\rho_t f_{m+t} + \dots + \rho_1 f_{m+1}+ f_m + \dots + f_1 +f_0 &\text{if } q>1.
\end{cases}
\]
This allows us to complete the proof, since any automorphism which maps $e_i$ to $f_i$ for each $i$ maps $w_0$ to $w_0'$.
\end{proof}

\subsection{$L$-space knots}
We specialise \eqref{eq:Viformula} to the case of $L$-space surgeries. A knot $K$ is said to be an {\em $L$-space knot} if $S^3_{p/q}(K)$ is an $L$-space for some $p/q \in \mathbb{Q}$. The knot Floer homology of an $L$-space knot is known to be determined by its Alexander polynomial, which can be written in the form
\[\Delta_K(t)=a_0 \sum_{i=1}^g a_i(t^i+t^{-i}),\]
where $g=g(K)$ and the non-zero values of $a_i$ alternate in sign and assume values in $\{\pm 1\}$ with $a_g=1$ \cite{Ozsvath2004genusbounds, Ozsvath2005Lensspace}. Given an Alexander polynomial in this form, we define its {\em torsion coefficients} by the formula
\[t_i(K) = \sum_{j\geq 1}ja_{|i|+j}.\]
\begin{rem}\label{rem:torsiondeterminespoly}
Observe that the torsion coefficients uniquely determine the Alexander polynomial. For $j\geq 1$, $a_j$ is determined by the relation
\[
a_{j}=t_{j-1}(K) -2t_{j}(K)+t_{j+1}(K).
\]
Since the Alexander polynomial is normalised so that $\Delta_K(1)=1$, this is also sufficient to determine the coefficient $a_0$.
\end{rem}

When $K$ is an $L$-space knot, the $V_i$ appearing in \eqref{eq:Viformula} satisfy $V_i=t_i(K)$ for $i\geq 0$ \cite{ozsvath2011rationalsurgery}. Thus if $S^3_{p/q}(K)$ is an $L$-space bounding a negative-definite sharp $4$-manifold $X$ with intersection form $Q_X$, then Theorem~\ref{thm:Gibbons} shows that $-Q_X$ embeds into $\mathbb{Z}^N$ for  $p/q$-changemaker lattice $L$, where
\[L=\langle w_0,\dots, w_l \rangle^\bot \subseteq \mathbb{Z}^{N}\]
and $w_0$ satisfies
\begin{equation}\label{eq:tiformula}
8t_{i}(K) = \min_{ \substack{ c\cdot w_0 \equiv 2i-n \bmod 2n \\ c \in \Char(\mathbb{Z}^{N})}} \norm{c} - N,
\end{equation}
for all $|i|\leq n/2$. If we write $w_0=\sigma_1 f_1 + \dots + \sigma_t f_t$, then Greene uses \eqref{eq:tiformula} to show that the genus $g(K)$ can be calculated by the formula \cite[Proposition~3.1]{Greene2010genusbounds}
\begin{equation}\label{eq:calculateg}
g(K)=\frac{1}{2}\sum_{i=1}^{t} \sigma_i(\sigma_i-1).
\end{equation}
This will allow us to prove Theorem~\ref{thm:Alexuniqueness}.

\begin{proof}[Proof of Theorem~\ref{thm:Alexuniqueness}]
Since every slope is known to be characterizing for the unknot \cite{Kronheimer2007lensspacsurgeries}, we can assume that $K$ is a non-trivial knot. In particular, we can assume that $g(K)>0$. Now suppose that $Y \cong S^3_{p/q}(K)$ is an $L$-space bounding a sharp 4-manifold $X$ with intersection form $Q_X$. Then the positive-definite lattice $-Q_X$ embeds into $\mathbb{Z}^{N}$ as a $p/q$-changemaker lattice,
\[L=\langle w_0, \dots , w_l \rangle^\bot \subseteq \mathbb{Z}^N,\]
where $N=b_2(X)+l+1$ and the torsion coefficients of $\Delta_K(t)$ satisfy the formula
\begin{equation}\label{eq:tiK}
t_i(K) = \min_{ \substack{ c\cdot w_0 \equiv n + 2i \bmod 2n \\ c \in \Char(\mathbb{Z}^{N})}} \norm{c} - N,
\end{equation}
for $|i|\leq n/2$. If we write $w_0$ in the form $w_0=\rho_t e_{t+m} + \dots + \rho_1 e_{m+1}+e_m + \dots + e_1$, then \eqref{eq:calculateg} becomes
\[2g(K)= \sum_{i=1}^t \rho_i(\rho_i-1).\]
Since $\rho_i \geq 2$ for all $i$, we have $\rho_i^2 \leq 2\rho_i(\rho_i -1)$. Thus we have
\begin{align}\begin{split}\label{eq:genusbound}
\sum_{i=1}^t \rho_i^2 + 2\rho_t &\leq 2\sum_{i=1}^t \rho_i(\rho_i-1) - \rho_t^2 +4\rho_t\\
&=4g(K) -(\rho_t-2)^2 + 4 \leq 4g(K)+4.
\end{split}\end{align}
If $K'\subset S^3$ is another knot such that $Y\cong S^3_{p/q}(K')$, then this gives another embedding of $-Q_X\cong L$ into $\mathbb{Z}^{b_2(X)+l+1}$ as a $p/q$-changemaker lattice
\[L'=\langle w'_0, \dots , w'_l \rangle^\bot,\]
where the torsion coefficients of $\Delta_{K'}(t)$ satisfy the formula
\begin{equation}\label{eq:tiK'}
t_i(K') = \min_{ \substack{ c\cdot w'_0 \equiv n + 2i \bmod 2n \\ c \in \Char(\mathbb{Z}^{N})}} \norm{c} - N.
\end{equation}
Combining the inequality \eqref{eq:genusbound} with the assumption $p/q\geq 4g(K)+4$ allows us to apply Lemma~\ref{lem:CMuniqueness}. This shows that there is an automorphism of $\mathbb{Z}^N$ mapping $w_0$ to $w_0'$. Since this automorphism will not alter the minimal values attained in \eqref{eq:tiK} and \eqref{eq:tiK'} for each $i$, this shows that the torsion coefficients satisfy $t_i(K)=t_i(K')$ for all $|i|\leq n/2$. Since $g(K)<n/2$, this implies that $t_i(K')=t_i(K)= 0$ for all $|i|\geq g(K)$. Thus we can conclude that $t_i(K')=t_i(K)$ for all $i$. As shown in Remark~\ref{rem:torsiondeterminespoly}, the torsion coefficients of $K$ and $K'$ determine their Alexander polynomials, so we have $\Delta_{K'}(t)=\Delta_{K}(t)$ and $g(K)=g(K')$, as required.
\end{proof}

\begin{rem}\label{rem:betterupperbound}
In the proof of Theorem~\ref{thm:Alexuniqueness}, the quantity $4g(K)+4$ arises as an upper bound to $B=\sum_{i=1}^t \rho_i^2 + 2\rho_t$, where $(\rho_1, \dots , \rho_t)$ are the stable coefficients appearing in the intersection form of the sharp 4-manifold $X$ bounding $S_{p/q}^3(K)$. In \cite{McCoy2017bounds}, it is shown that the tuple $(\rho_1, \dots , \rho_t)$ is independent of the manifold $X$ and is, in fact, an invariant of the knot $K$. Given this fact, we could replace $4g(K)+4$ in Theorem~\ref{thm:Alexuniqueness} by the quantity $B$. In general, $B$ will be smaller that $4g(K)+4$. For example, if $K$ is the torus knot $T_{r,s}$, then one can show that
\[B\leq rs+2\min\{r,s\}-2.\]
However although $B$ is in general a better bound, it is not an improvement in all cases. Namely, for the torus knots $T_{2,s}$ one can show that
\[B=4g(T_{2,s})+4=2s+2.\]
\end{rem}

\section{Characterizing slopes}\label{sec:charslopes}
In this section, we prove Theorem~\ref{thm:charslopes}. Our proof follows the one given by Ni and Zhang. We obtain our improvement through the following lemma.
\begin{lem}\label{lem:torusgenus}
For the torus knot $T_{r,s}$ with $r>s>1$, any knot $K\subset S^3$ satisfying
\[S^3_{p/q}(K)\cong S^3_{p/q}(T_{r,s}),\]
for some $p/q\geq 4g(T_{r,s})+4$, has genus $g(K)=g(T_{r,s})$ and Alexander polynomial $\Delta_{K}(t)=\Delta_{T_{r,s}}(t)$.
\end{lem}
\begin{proof}
Since $r>s>1$ and $p/q\geq 4g(T_{r,s})+4>2g(T_{r,s})-1$, it follows that $S^3_{p/q}(T_{r,s})$ is an $L$-space. Since $S^3_{rs-1}(T_{r,s})$ is a lens space \cite{Moser1971elementary}, Ozsv{\'a}th and Szab{\'o} show that it bounds a sharp 4-manifold \cite{Ozsvath2003plumbed, Ozsvath2005branched}. Therefore, since $p/q>rs-1$, Theorem~\ref{thm:sharpextension} shows that $S^3_{p/q}(T_{r,s})$ also bounds a sharp 4-manifold. This allows us to apply Theorem~\ref{thm:Alexuniqueness}, which gives the desired conclusion.
\end{proof}
\begin{rem}
It is actually possible to exhibit a sharp manifold bounding $S^3_{p/q}(T_{r,s})$ explicitly. Since the manifold $S^3_{p/q}(T_{r,s})$ is a Seifert-fibred space with base orbifold $S^2$ with at most 3 exceptional fibres \cite{Moser1971elementary}, it bounds a plumbed 4-manifold. For $p/q\geq rs-1$, one can find such a plumbing which is negative-definite and sharp.
\end{rem}

Using results of Agol \cite{Agol2000BoundsI}, Cao-Meyerhoff \cite{Cao2001cusped} and Lackenby \cite{Lackenby2003Exceptional}, Ni and Zhang obtain a restriction on exceptional slopes of a hyperbolic knot.
\begin{prop}[Lemma~2.2, \cite{Ni2014characterizing}]\label{prop:hyperbolicbound}
Let $K\subseteq S^3$ be a hyperbolic knot. If
\[|p|\geq 10.75(2g(K)-1),\]
then $S^3_{p/q}(K)$ is hyperbolic.
\end{prop}
Combining this with work of Gabai, they show that it is not possible for surgery of sufficiently large slope on a satellite knot and a torus knot to yield the same manifold.
\begin{lem}\label{lem:satellitebound}
If $K$ is a knot such that $S^3_{p/q}(K)\cong S^3_{p/q}(T_{r,s})$ for $r>s>1$ and $p/q\geq 10.75(2g(T_{r,s})-1)$, then $K$ is not a satellite.
\end{lem}
\begin{proof}
If $K$ is a satellite knot, then let $R\subset S^3 \setminus K$ be an incompressible torus. This bounds a solid torus $V\subseteq S^3$ which contains $K$. Let $K'$ be the core of the solid torus $V$. By choosing $R$ to be ``innermost'', we may assume that $K'$ is not a satellite. This means that $K'$ is either a torus knot or it is hyperbolic \cite{Thurston1982KleinianGroups}. Since $S^3_{p/q}(T_{r,s})$ contains no incompressible tori and is irreducible, it follows from the work of Gabai that $V_{p/q}(K)$ is again a solid torus and $K$ is either a 1-bridge knot or a torus knot in $V$ \cite{Gabai1989solidtori}. In either case, this is a braid in $V$ and we have
$S_{p/q}^3(K)\cong S_{p/q'}^3(K')$ where $q'=qw^2$ and $w>1$ is the winding number of $K$ in $V$.

Since $p\geq 10.75(2g(K)-1)$, Proposition~\ref{prop:hyperbolicbound} shows that $K'$ cannot be hyperbolic. Thus we may assume that $K'$ is a torus knot, say $K'=T_{m,n}$. Since $S^3_{p/q}(T_{r,s})$ is an $L$-space and $p/q'>0$ we have $m,n>1$. The manifold $S_{p/q}^3(T_{r,s})\cong S_{p/q'}^3(T_{m,n})$ is Seifert fibred over $S^2$ with exceptional fibres of order $\{r,s,p-qrs\}= \{m,n, |p-q'mn|\}$. Hence we can assume $m=r$. By Lemma~\ref{lem:torusgenus}, we have $\Delta_{T_{r,s}}(t)=\Delta_K(t)$. However, since $K$ is a satellite, its Alexander polynomial takes the form $\Delta_K(t)=\Delta_C(t)\Delta_{K'}(t^w)$, where $C$ is the companion knot of $K$. In particular, we have $g(K')<g(T_{r,s})$ and, consequently, $n<s$. Comparing the orders of the exceptional fibres again, this implies that $n=p-qrs$. However, we have
\begin{align*}
p-rsq&\geq 9.75q(rs-r-s)-q(r+s)\\
     &\geq 9.75(\max\{r,s\}-2) -(2\max\{r,s\}-1)\\
     &= 7.75\max\{r,s\} -18.5\\
     &\geq \max\{r,s\},
\end{align*}
where the last inequality holds because we have $\max\{r,s\}\geq 3$. This is a contradiction and shows that $K'$ cannot be a torus knot. Thus we see that $K$ cannot be a satellite knot.
\end{proof}

\begin{proof}[Proof of Theorem~\ref{thm:charslopes}]
Suppose that $K$ is a knot in $S^3$ with $Y\cong S^3_{p/q}(K) \cong S^3_{p/q}(T_{r,s})$ for $p/q\geq 10.75(rs-r-s)$. Lemma~\ref{lem:torusgenus} shows that $g(K)=g(T_{r,s})$ and $\Delta_K(t)=\Delta_{T_{r,s}}(t)$. Since $Y$ is not hyperbolic, Proposition~\ref{prop:hyperbolicbound} shows that $K$ is not a hyperbolic knot. Lemma~\ref{lem:satellitebound} shows that $K$ is not a satellite knot. Therefore, it follows that $K$ is a torus knot. Since two distinct torus knots have the same Alexander polynomial only if they are mirrors of one another, $K$ is either $T_{r,s}$ or $T_{-r,s}$. As $K$ admits positive $L$-space surgeries, it follows that $K=T_{r,s}$, as required.
\end{proof}

\bibliographystyle{alpha}
\bibliography{master}

\begin{thebibliography}{KMOS07}

\bibitem[Ago00]{Agol2000BoundsI}
I.~Agol.
\newblock Bounds on exceptional {D}ehn filling.
\newblock {\em Geom. Topol.}, 4:431--449, 2000.

\bibitem[CM01]{Cao2001cusped}
Chun Cao and G.~Robert Meyerhoff.
\newblock The orientable cusped hyperbolic {$3$}-manifolds of minimum volume.
\newblock {\em Invent. Math.}, 146(3):451--478, 2001.

\bibitem[Don83]{donaldson1983application}
S.~K. Donaldson.
\newblock An application of gauge theory to four-dimensional topology.
\newblock {\em J. Differential Geom.}, 18(2):279--315, 1983.

\bibitem[Gab89]{Gabai1989solidtori}
David Gabai.
\newblock Surgery on knots in solid tori.
\newblock {\em Topology}, 28(1):1--6, 1989.

\bibitem[Gib15]{Gibbons2013deficiency}
Julian Gibbons.
\newblock Deficiency symmetries of surgeries in ${S}^3$.
\newblock {\em International Mathematics Research Notices},
  2015(22):12126--12151, 2015.

\bibitem[Gre13]{Greene2013Realization}
Joshua Greene.
\newblock The lens space realization problem.
\newblock {\em Ann. of Math. (2)}, 177(2):449--511, 2013.

\bibitem[Gre14]{Greene2014_3Braid}
Joshua~Evan Greene.
\newblock Donaldson's theorem, {H}eegaard {F}loer homology, and knots with
  unknotting number one.
\newblock {\em Adv. Math.}, 255:672--705, 2014.

\bibitem[Gre15]{Greene2010genusbounds}
Joshua~Evan Greene.
\newblock L-space surgeries, genus bounds, and the cabling conjecture.
\newblock {\em J. Differential Geom.}, 100(3):491--506, 2015.

\bibitem[KMOS07]{Kronheimer2007lensspacsurgeries}
P.~Kronheimer, T.~Mrowka, P.~Ozsv{\'a}th, and Z.~Szab{\'o}.
\newblock Monopoles and lens space surgeries.
\newblock {\em Ann. of Math. (2)}, 165(2):457--546, 2007.

\bibitem[Lac03]{Lackenby2003Exceptional}
Marc Lackenby.
\newblock Exceptional surgery curves in triangulated 3-manifolds.
\newblock {\em Pacific J. Math.}, 210(1):101--163, 2003.

\bibitem[McC15]{McCoy2015noninteger}
Duncan McCoy.
\newblock Non-integer surgery and branched double covers of alternating knots.
\newblock {\em J. Lond. Math. Soc. (2)}, 92(2):311--337, 2015.

\bibitem[McC17]{McCoy2017bounds}
Duncan McCoy.
\newblock Bounds on alternating surgery slopes.
\newblock {\em Algebr. Geom. Topol.}, 17(5):2603--2634, 2017.

\bibitem[Mos71]{Moser1971elementary}
Louise Moser.
\newblock Elementary surgery along a torus knot.
\newblock {\em Pacific J. Math.}, 38:737--745, 1971.

\bibitem[NW15]{ni2010cosmetic}
Yi~Ni and Zhongtao Wu.
\newblock Cosmetic surgeries on knots in {$S^3$}.
\newblock {\em J. Reine Angew. Math.}, 706:1--17, 2015.

\bibitem[NZ14]{Ni2014characterizing}
Yi~Ni and Xingru Zhang.
\newblock Characterizing slopes for torus knots.
\newblock {\em Algebr. Geom. Topol.}, 14(3):1249--1274, 2014.

\bibitem[OS03a]{Ozsvath2003Absolutely_graded}
Peter Ozsv{\'a}th and Zolt{\'a}n Szab{\'o}.
\newblock Absolutely graded {F}loer homologies and intersection forms for
  four-manifolds with boundary.
\newblock {\em Adv. Math.}, 173(2):179--261, 2003.

\bibitem[OS03b]{Ozsvath2003plumbed}
Peter Ozsv{\'a}th and Zolt{\'a}n Szab{\'o}.
\newblock On the {F}loer homology of plumbed three-manifolds.
\newblock {\em Geom. Topol.}, 7:185--224 (electronic), 2003.

\bibitem[OS04]{Ozsvath2004genusbounds}
Peter Ozsv{\'a}th and Zolt{\'a}n Szab{\'o}.
\newblock Holomorphic disks and genus bounds.
\newblock {\em Geom. Topol.}, 8:311--334, 2004.

\bibitem[OS05a]{Ozsvath2005Lensspace}
Peter Ozsv{\'a}th and Zolt{\'a}n Szab{\'o}.
\newblock On knot {F}loer homology and lens space surgeries.
\newblock {\em Topology}, 44(6):1281--1300, 2005.

\bibitem[OS05b]{Ozsvath2005branched}
Peter Ozsv{\'a}th and Zolt{\'a}n Szab{\'o}.
\newblock On the {H}eegaard {F}loer homology of branched double-covers.
\newblock {\em Adv. Math.}, 194(1):1--33, 2005.

\bibitem[OS08]{Ozsvath2008integersurgery}
Peter Ozsv{\'a}th and Zolt{\'a}n Szab{\'o}.
\newblock Knot {F}loer homology and integer surgeries.
\newblock {\em Algebr. Geom. Topol.}, 8(1):101--153, 2008.

\bibitem[OS11]{ozsvath2011rationalsurgery}
Peter Ozsv{\'a}th and Zolt{\'a}n Szab{\'o}.
\newblock Knot {F}loer homology and rational surgeries.
\newblock {\em Algebr. Geom. Topol.}, 11(1):1--68, 2011.

\bibitem[OS12]{Owens2012negdef}
Brendan Owens and Sa{\v{s}}o Strle.
\newblock Dehn surgeries and negative-definite four-manifolds.
\newblock {\em Selecta Math. (N.S.)}, 18(4):839--854, 2012.

\bibitem[OS16]{Owens2016Immersed}
Brendan Owens and Sa\v{s}o Strle.
\newblock Immersed disks, slicing numbers and concordance unknotting numbers.
\newblock {\em Comm. Anal. Geom.}, 24(5):1107--1138, 2016.

\bibitem[Thu82]{Thurston1982KleinianGroups}
William~P. Thurston.
\newblock Three-dimensional manifolds, {K}leinian groups and hyperbolic
  geometry.
\newblock {\em Bull. Amer. Math. Soc. (N.S.)}, 6(3):357--381, 1982.

\end{thebibliography}
\end{document}